\documentclass{amsart}
\usepackage{amssymb, amsmath}
\usepackage[dvips]{graphicx}
\usepackage{amsfonts}
\usepackage{latexsym}
\usepackage{color}
\newtheorem{theorem}{Theorem}	
\newtheorem{lemma}{Lemma}[section]		
\newtheorem{corollary}{Corollary}		
\newtheorem{proposition}{Proposition}

\begin{document}
\title[Self-dual polytope and self-dual smooth Wulff shape]
{Self-dual polytope and self-dual smooth Wulff shape}
\author{
Huhe Han}
\address{College of Science, Northwest Agriculture and Forestry University, China}
\email{han-huhe@nwafu.edu.cn}
\subjclass[2010]{52A55}
\keywords{Wulff shape, Dual Wulff shape, Self-dual Wulff shape, spherical convex polytope,  
constant width, 
Spherical polar set}
\begin{abstract}
For any Wulff shape $\mathcal{W}$, 
its dual Wulff shape and spherical Wulff shape $\widetilde{\mathcal{W}}$ 
can be defined naturally.
A self-dual Wulff shape is a Wulff shape equaling its dual Wulff shape exactly.   
In this paper, we show that 
if a spherical convex polytope $P$ is of constant width $\delta$, then $\delta=\pi/2$. 
As an application of this fact, we prove that a polytope Wulff shape is self-dual 
if and only if its spherical Wulff shape is a spherical convex body of constant width.
We also prove that 
a smooth Wulff shape is self-dual if and only if 
for any interior point $P$ of $\widetilde{\mathcal{W}}$ and 
for any point $Q$ of the intersection of the boundary of $\widetilde{\mathcal{W}}$ and the graph of its spherical support function (with respect to $P$),
the image of $Q$ under the spherical blow-up (with respect to $P$) is always a boundary point of 
$\widetilde{\mathcal{W}}$.
\end{abstract}
\maketitle

\section{Introduction}
Let $S^n$ and $\mathbb{R}_+$ be $n$ dimensional unit sphere in $\mathbb{R}^{n+1}$ and the set consisting of positive real numbers respectively. 
Let $\gamma:S^n\to \mathbb{R}_+$ be a continuous function. 
Set $\Gamma_{\gamma,\theta}=\{x\in\mathbb{R}^{n+1}|x\cdot\theta\leq \gamma(\theta)\}$, 
where the dot in the center stands for the dot product of two vectors $x, \theta$ of $\mathbb{R}^{n+1}$.
The {\it Wulff shape}  associated with support function $\gamma$, denoted by $\mathcal{W}_\gamma$, is the following set:
\[
\bigcap_{\theta\in S^n}\Gamma_{\gamma,\theta}.
\]
By definition,  Wulff shape is compact, convex and contains the origin of $\mathbb{R}^{n+1}$ as an interior.
The Wulff shape was first introduced by G.Wulff in \cite{wulff}, is known as a geometric model of a crystal at equilibrium (see for instance \cite {crystalbook,taylor}). 
In \cite{nishimurasakemi2}, an equivalent definition of Wulff shape 
is presented, which is defined as the composition of the following mappings.
\par (1)
Let $Id: \mathbb{R}^{n+1}\to \mathbb{R}^{n+1}\times \{1\}\subset \mathbb{R}^{n+2}$
be the mapping defined by
\[
Id(x)=(x,1).
\]
\par (2)
Let $N=(0,\ldots,0,1)\in \mathbb{R}^{n+2}$ denote the north pole of $S^{n+1}$,
and let $S_{N,+}^{n+1}$ denote the north open hemisphere of $S^{n+1}$, that is,
\[
S_{N,+}^{n+1}=\{Q\in S^{n+1}| N\cdot Q>0\}.
\]
The {\it central projection relative to} $N$, denoted by
$\alpha_N: S_{N,+}^{n+1}\to \mathbb{R}^{n+1}\times \{1\}$, is defined by
\[
\alpha_N\left(P_1, \ldots, P_{n+1}, P_{n+2}\right)
=
\left(\frac{P_1}{P_{n+2}}, \ldots, \frac{P_{n+1}}{P_{n+2}}, 1\right).
\]
We call the spherical convex body $\widetilde{W}_\gamma=\alpha^{-1}(Id(\mathcal{W_\gamma}))$ is {\it the spherical Wulff shape of $\mathcal{W_\gamma}$}.
\par (3)
The {\it spherical blow-up} (with respect to $N$) $\Psi_N:S^{n+1}-\{\pm N\}\to {\color{black}S_{N, +}^{n+1}}$, 
is defined by 
\[
\Psi_N(P)=\frac{1}{\sqrt{1-(N\cdot P)^2}}(N-(N\cdot P)P).    
\]
The mapping $\Psi_N$, which was first introduced in \cite{nishimura}. 
It is obvious that  
\[
|\Psi_N(P)P|=\pi/2\]
for any point $P\in S^{n+1}-\{\pm N\}$.
\par 
(4) 
For any point $P \in S^{n+1}$, let $H(P)$
be the hemisphere centered at $P$,
\[
H(P)=\{Q\in S^{n+1}\mid P\cdot Q\geq 0\},
\]
where the dot in the center stands for the scalar product of two vectors 
$P, Q\in \mathbb{R}^{n+2}$.     
For any {\color{black}non-empty} subset $\widetilde{W}\subset S^{n+1}$, the {\it spherical polar set of $\widetilde{W}$}, denoted by 
$\widetilde{W}^\circ$, is defined as follows: 
\[
\widetilde{W}^\circ = \bigcap_{P\in \widetilde{W}}H(P).
\]   
\begin{proposition}[\cite{nishimurasakemi2}]\label{sphericalmethod}
Let $\gamma: S^{n}\to \mathbb{R}_+$ be a continuous function.    
Set 
\[
\mbox{\rm graph}(\gamma)=\{(\theta, \gamma(\theta))\in \mathbb{R}^{n+1}-\{0\}\; |\; \theta\in S^n\},
\]
where 
$(\theta, \gamma(\theta))$ is the polar plot expression for a point of $\mathbb{R}^{n+1}-\{0\}$.   
Then, $\mathcal{W}_\gamma$ {\color{black}is} characterized as follows:   
\[
\mathcal{W}_\gamma = 
Id^{-1}\circ \alpha_{{}_{N}}\left(\left(\Psi_N\circ \alpha_{{}_{N}}^{-1}\circ 
Id\left(\mbox{\rm graph}(\gamma)\right)\right)^\circ\right).
\]
\end{proposition}
For any Wulff shape $\mathcal{W}_\gamma$, 
the set
\[
Id^{-1}\circ \alpha_{{}_{N}}\left(\left(\alpha_{{}_{N}}^{-1}\circ 
Id\left(\mathcal{W}_\gamma\right)\right)^\circ\right), 
\] 
is called the {\it dual Wulff shape} of $\mathcal{W}_\gamma$, denoted by $\mathcal{DW}_\gamma$(\cite{nishimurasakemi2}).
A Wulff shape $\mathcal{W}_\gamma$ is called {\it self-dual} if $\mathcal{W}_\gamma=\mathcal{DW}_\gamma$, 
namely, $\mathcal{W}_\gamma$ and its dual Wulff shape are same convex bodies. 
In \cite{hnjmsj},  a self-dual Wulff shape is characterized as  its spherical Wulff shape $\alpha^{-1}(Id(\mathcal{W_\gamma}))$ 
is of constant width $\pi/2$(definitions of constant width and constant diamter spherical convex body see Section 2).
Applying this fact, in \cite{hwam,Lassak20},  states a characterization of self-dual Wulff shape as constant 
diameter convex body of $\pi/2$.
\begin{theorem}[\cite{hwam, Lassak20}]\label{generallassak1}
Let $\mathcal{W}_{\gamma}$ be a Wulff shape. Then the following statements are equivalent: 
\begin{enumerate} 
\item the Wulff shape $\mathcal{W}_{\gamma}$ is self-dual;
\item the spherical Wulff shape of $\mathcal{W}_{\gamma}$ is of constant width $\pi/2$;
\item the spherical Wulff shape of $\mathcal{W}_{\gamma}$ is of constant diameter $\pi/2$.
\end{enumerate}
\end{theorem}
By Lemma \ref{constantpolytope} below, we know that the condition $\pi/2$ of Theorem \ref{generallassak1} is unnecessary if Wulff shape $\mathcal{W}_{\gamma}$ is a convex polytope (Proposition \ref{generallassak}). 
\begin{lemma}\label{constantpolytope}
Let $\mathcal{P}$ be a spherical convex polytope of constant width $\delta$. Then $\delta=\pi/2$. 
\end{lemma}
\begin{proposition}\label{generallassak}
Let Wulff shape $\mathcal{W}_{\gamma}$ be a polytope. Then the following statements are equivalent: 
\begin{enumerate} 
\item the Wulff shape $\mathcal{W}_{\gamma}$ is self-dual;
\item the spherical Wulff shape of $\mathcal{W}_{\gamma}$ is of constant width;
\item the spherical Wulff shape of $\mathcal{W}_{\gamma}$ is of constant diameter.
\end{enumerate}
\end{proposition}
We present a local characterization of smooth self-dual Wulff shape as following:
\begin{theorem}\label{main}
Let $\mathcal{W}$ be a smooth Wulff shape. Then the following two statements are equivalent. 
\begin{enumerate}
\item $\mathcal{W}$ is a self-dual Wulff shape.
\item The spherical Wulff shape $\alpha^{-1}(Id(\mathcal{W_\gamma}))$ of $\mathcal{W}$, namely $\widetilde{\mathcal{W}}$,  has the property that 
$\Psi_{M}(P)$ is a point of $\partial\widetilde{\mathcal{W}}$
for any interior point $M$ of $\widetilde{\mathcal{W}}$ and 
any point $P$ of $\partial \widetilde{\mathcal{W}}\cap \Psi_M\left((\partial\widetilde{\mathcal{W}})^{\circ}\right)$.
\end{enumerate}
Here $\partial K$ denotes the boundary of $K\subset S^{n+1}$. 
\end{theorem}
Geometrically, Theorem \ref{main} means 
a smooth Wulff shape is self-dual if and only if 
for any interior points $P$ of $\widetilde{\mathcal{W}}$ and 
for any point $Q$ of the intersection of the boundary of $\widetilde{\mathcal{W}}$ and the graph of its spherical support function (with respect to $P$),
the image of $Q$ under the spherical blow-up (with respect to $P$) is always a boundary point of 
$\widetilde{\mathcal{W}}$.

\par
This paper is organized as follows. 
 In Section 2, we prepare several properties of width and diameter of spherical convex bodies for proofs of 
 Proposition \ref{generallassak} and Theorem \ref{main}.
The proofs of Proposition \ref{generallassak} and Theorem \ref{main} are given in Section 3, Section 4 respectively.
\section{Width and diameter of spherical convex bodies}
A subset $X\subset S^{n+1}$ is said to be {\it hmispherical} if there exists a
point $P$ of  $S^{n+1}$ such that $H(P)\cap X=\emptyset$.
For any $P,Q\in S^{n+1} (P\neq -Q)$, the arc connecting $P$ and $Q$, denoted by $PQ$,  is defined by
the shorter part of the great circle containing these points, namely,
\[
PQ=\left\{
\frac{(1-t)P+tQ}{||(1-t)P+tQ||}\in S^{n+1} \mid t\in [0,1]
\right\},
\]
where $||\cdot ||$ denotes the standard $(n+2)-$dimensional Euclidean norm.
 The spherical distance $|PQ|$ of $P$ and $Q$ can be given by
$\arccos(P\cdot Q)$.
\par
A hemispherical subset $X\subset S^{n+1}$ is said to be {\it spherical convex}
if the arc $PQ$ is a subset of  $X$ for any $P, Q\in X$.
Let $C\subset S^{n+1}$ be a spherical convex body. 
Let $P$ be a point on the boundary of $C$. 
If the hemisphere $H(Q)$ contains $C$ and $P$ is a point of the intersection $\partial C \cap \partial H(Q)$, then we say $H(Q)$ {\it supports} $C$ at $P$, and $H(Q)$ is a {\it supporting hemisphere} of $C$ at $P$. 
A spherical convex body $C$ is said to be {\it smooth} if  for any boundary point $P$ of $C$,
 exactly one hemisphere supports $C$ at $P$.
If hemispheres $H(P)$ and $H(Q)$ of $S^n$ are different and not opposite ($P\neq -Q$), 
then the intersection $H(P)\cap H(Q)$ is called a {\it lune} of $S^{n+1}$. 
The {\it thickness of lune} $H(P)\cap H(Q)$, denoted by $\Delta (H(P)\cap H(Q))$, is given by 
$\pi-|PQ|$. That is  
\[
\Delta (H(P)\cap H(Q))=\pi-|PQ|. 
\]
If $H(P)$ is a supporting hemisphere of a spherical convex body $C$, 
{\it the width of} $C$ with respect to $H(P)$, denoted by $\mbox{width}_{H(P)} (C)$, is defined by 
\[
\mbox{min}\{\Delta(H(P)\cap H(Q))| H(Q)\ \mbox{supports}\ C\}.
\]
Namely,
$
\mbox{width}_{H(P)} (C) =\mbox{min}\{\Delta(H(P)\cap H(Q))| H(Q)\ \mbox{supports}\ C\}.
$
We say that the spherical convex body $C$ is {\it of constant width},
if all widths of $C$ with respect to any supporting hemispheres are equal. 
Following \cite{Lassak15}, 
we define {\it the thickness} of a convex body $C\subset S^n$, denoted by $\Delta(C)$,  as 
the minimum of $\mbox{width}_{H(P)} (C)$ over all supporting hemispheres $H(P)$ of $C$.
That is 
\[
\Delta(C) = \mbox{min}\{\mbox{width}_K(C)| K\ \mbox{is\ a\ supporting\ hemisphere\ of\ }C\}.
\]
The {\it diameter} of a spherical convex body $C$, denoted by $\mbox{diam}(C)$, is defined by
$\mbox{max}\{|P Q|| P,Q \in C\}$. 
Following \cite{LM18}, we say that the spherical convex body $C$  is of {\it constant diameter} $\delta>0$ , if the diameter of $C$ is $\delta$ , and for every point
$P$ on the boundary of $C$ there exists a point $Q$ of $C$ such that $|P Q| = \delta$. 
More details on width, thickness and diameter of spherical convex bodies, see for instance 
\cite{perimeter20,  Lassak20-1,Lassak20, Lassak22, Musidlak20}.
\par
To prove Theorem \ref{main}, we require the following Lemmas.
\begin{lemma}[\cite{LM18}]\label{diameterorthognal}
Assume that the diameter of a convex body $C\subset S^n$ is realized for
points $P$ and $Q$. 
The hemisphere $K$ orthogonal to $PQ$ at $P$ and containing $Q\in K$ supports $C$.
\end{lemma}
Let $C$ be a hemispherical subset of $S^{n+1}$. 
The {\it spherical convex hull of} $C$, denoted by $\mbox{s-conv}(C)$, is defined by
\[
\mbox{s-conv}(C) =\left\{
\frac{\sum_{i=1}^k t_iP_i}{\mid\mid\sum_{i=1}^k t_iP_k\mid\mid} |
\sum_{i=1}^kt_i=1, t_i\geq 0, P_1,\dots, P_k\in C, k\in \mathbb{N}
\right\}.
\]
We say that 
$\mbox{s-conv}(\{P_1,\dots, P_k\})$ is
{\it the spherical (convex) polytope generated by} $P_1,\dots, P_k$,
 if it has an interior point and
$P_1,\dots, P_k$ is a hemispherical  set.
\begin{lemma}[\cite{nishimurasakemi2}]\label{lemmapoly}
\begin{enumerate}
\item Let $\mathcal{P}\subset S_{N,+}^{n+1}$
be the spherical polytope generated by $P_1,\dots, P_k$.
Then, $\alpha_N (X)$ is the polytope generated by $\alpha_N (P_1),\dots, \alpha_N (P_k)$.
\item Let e $X\subset \mathbb{R}^{n+1}\times \{1\}$ 
be the polytope generated by $(p_1, 1),\dots,(p_k, 1)$. 
Then, $\alpha_N^{-1}(X)$
is the spherical polytope generated by $\alpha_N^{-1}((p1, 1)), . . . , \alpha_N^{-1}((p_k, 1))$.
\end{enumerate}
\end{lemma}
\section{ Proof of Proposition \ref{generallassak}:}
\begin{lemma}[\cite{hwam}]\label{dualconstantwidth} 
Let $C$ be a spherical convex body in $S^{n+1}$, and $0<\delta<\pi$. The
following two assertions are equivalent:
\begin{enumerate}
\item $C$ is of constant width $\delta$.
\item $C^\circ$ is of constant width $\pi-\delta$.
\end{enumerate}
\end{lemma}
\begin{lemma}[\cite{LM18}]\label{strictlyconvex}
Every spherical convex body of constant width smaller than $\pi/2$ on $S^{n+1}$ is strictly convex.
\end{lemma}
Since a Wulff shape is self-dual if and only if its spherical Wulff shape is of constant width $\pi/2$ 
(or equivalently, its spherical Wulff shape is of constant diameter $\pi/2$), 
in order to prove Proposition \ref{generallassak}, by Lemma \ref{lemmapoly}, it is sufficient to 
prove that a spherical convex polytope $\mathcal{P}$ of constant width $\delta$, then $\delta=\pi/2$(Lemma \ref{constantpolytope}). 
 By Lemma \ref{strictlyconvex}, 
a spherical convex polytope $\mathcal{P}$ is of constant width $\delta$, then $\delta\geq \pi/2$. 
Suppose that $\delta>\pi/2$. 
By Lemma \ref{dualconstantwidth}, its spherical polar set $\mathcal{P}^\circ$ is a spherical convex body of constant width
$\pi-\delta<\pi/2$. 
On the other hand, 
by assumption, $\mathcal{P}^\circ$ is a spherical polytope.
This contradicts Lemma \ref{strictlyconvex}.
\section{Proof of Theorem \ref{main}}
\begin{lemma}\label{lemmasmooth}
Let $\widetilde{\mathcal{W}}$ be a smooth spherical convex body  and 
let $H(Q)$ be a hemisphere supports $\widetilde{\mathcal{W}}$ at $P$. 
Then the intersection of $PQ$ and the interior of $\widetilde{\mathcal{W}}$ is non empty.
\end{lemma}
\begin{proof}
Suppose that the  intersection of $PQ$ and the interior of $\widetilde{\mathcal{W}}$ is empty.
Without loss of generality, we may assume that $Q=N$ is the north pole of $S^{n+1}$ and $P=(1, 0,\dots, 0)$. 
Let $\mathcal{\pi}_Q: H(Q)\backslash Q\to \partial H(Q)$ be the spherical projection defined by 
\[
\mathcal{\pi}_Q(R)=\frac{\frac{1}{1-R_{n+2}}R+(1-\frac{1}{1-R_{n+2}})Q}
{||\frac{1}{1-R_{n+2}}R+(1-\frac{1}{1-R_{n+2}})Q||},
\]
where $R=(R_1, \dots, R_{n+2})\in H(Q)\backslash Q$. 
Namely, $\mathcal{\pi}_Q(R)$ is an end point of the quarter great circle $Q\mathcal{\pi}_Q(R)$ contains arc $QR$. 
Identifying $\mathbb{R}^{n+1}\times \{0\}$ with $\mathbb{R}^{n+1}$ and $\partial H(Q)$ with $S^n$ of $\mathbb{R}^{n+1}$. 
Then, by assumption, $\mathcal{\pi}_Q(P)=P$ is a boundary point of $\mathcal{\pi}_Q(C)$, 
and $\mathcal{\pi}_Q(C)$ is a spherical convex body of $S^n$. 
Let $\mathcal{H}^{n}\subset S^n$ be the supporting hemisphere of $\mathcal{\pi}_Q(C)$  at $P$ 
and let $\mathcal{H}^{n+1}\subset S^{n+1}$ be the hemisphere contains 
$\mathcal{\pi}_Q^{-1}(\mathcal{H}^n)$. 
Moreover, by definition of $\mathcal{\pi}_Q$, the intersection of 
$\mathcal{H}^{n+1}$ and the interior of $C$ is empty. 
Hence, $\mathcal{H}^{n+1}$ is the supporting hemisphere of $C$ at $P$. 
Since $Q$ is a point of $\mathcal{H}^{n+1}$, 
it follows that $\mathcal{H}^{n+1}$ and $H(Q)$ are different supporting hemispheres 
of $C$ at $P$.
This result contradicts the assumption that $C$ is a smooth spherical convex body.
\end{proof}
\begin{lemma}\label{leproperty}
Let $\widetilde{\mathcal{W}}$ be a smooth spherical Wulff shape, and has the property that 
for any interior point $M$ of $\widetilde{\mathcal{W}}$ and any point $P$ of $\partial \widetilde{\mathcal{W}}\cap \Psi_M\left(\partial\widetilde{\mathcal{W}}^{\circ}\right)$ 
it follows that $\Psi_{M}(P)$ is a point of $\partial\widetilde{\mathcal{W}}$. 
Then {\rm diam($\widetilde{\mathcal{W}}$)=$\pi/2$}.
\end{lemma}
\begin{proof}
By Proposition \ref{sphericalmethod}, we know that if $N=M$ then
\[
\alpha_M\left(\partial \widetilde{\mathcal{W}}\cap \Psi_M\left((\partial\widetilde{\mathcal{W}})^{\circ}\right)\right)
=
\partial \mathcal{W}_\gamma\cap \mbox{\rm graph}(\gamma).
\]
This means 
the intersection 
$\partial \widetilde{\mathcal{W}}\cap \Psi_M\left((\partial\widetilde{\mathcal{W}})^{\circ}\right)$ always non-empty for any interior point $M$ of $\widetilde{\mathcal{W}}$. 
Then, by assumption, it follows that {\rm diam($\widetilde{\mathcal{W}})\geq \pi/2$}.
Let $P, Q$ be two points of $\partial\widetilde{\mathcal{W}}$ 
such that $|PQ|$={\rm diam($\widetilde{\mathcal{W}}$)}. 
By Lemma \ref{diameterorthognal}, the hemisphere $K$ orthogonal to $PQ$ at $P$ and containing $Q\in K$ supports $\widetilde{\mathcal{W}}$.
Set $K=H(R)$.
Since $\widetilde{\mathcal{W}}$ is a subset of $H(R)$, it follows that 
\[
R\in \bigcap_{Q\in\widetilde{\mathcal{W}}}H(Q)=\widetilde{\mathcal{W}}^{\circ}.
\]
By Proposition \ref{sphericalmethod} and Lemma \ref{lemmasmooth}, there exists a point $M$ of the intersection of $PR$ and interior of $\widetilde{\mathcal{W}}$ such that
\[
P \in \partial \widetilde{\mathcal{W}}\cap \Psi_{M}\left((\partial\widetilde{\mathcal{W}})^{\circ}\right).
\]
By assumption, it follows that $\Psi_{M}(P)=R$ is a point of $\partial\widetilde{\mathcal{W}}$. 
This implies $R=Q$.
Thus, we have 
\[
{\rm diam(\widetilde{\mathcal{W}})}=|PQ|=|PR|=|P\Psi_{M}(P)|=\pi/2.
\]
\end{proof}
We are now in position to prove Theorem \ref{main}.
\begin{proof}
First, we proof that the Wulff shape $\mathcal{W}$ has the property (2) of Theorem \ref{main} under the assumption that $\mathcal{W}$ is a self-dual Wulff shape.
Let $M$ be an interior point of $\widetilde{\mathcal{W}}$. Then for any point 
$P$ of $\partial \widetilde{\mathcal{W}}\cap \Psi_M\left((\partial\widetilde{\mathcal{W}})^{\circ}\right)$,
by assumption, it follows that 
\[
\Psi_{M}(P)\in \Psi_{M}\left(\Psi_{M}(\partial\widetilde{\mathcal{W}}^{\circ})\right)=\partial\widetilde{\mathcal{W}}^{\circ}=\partial\widetilde{\mathcal{W}}.
\]
\par
\indent
Next, we proof that the Wulff shape $\mathcal{W}$ is a self-dual Wulff shape under the assumption that 
$\widetilde{\mathcal{W}}$ has the property (2) of Theorem \ref{main}. By Theorem \ref{generallassak1}, 
it is sufficient to prove that $\widetilde{\mathcal{W}}$ is of constant width $\pi/2$. 
First, we prove ${\rm width}_K(\widetilde{\mathcal{W}})\leq \pi/2$ for any supporting hemisphere $K$ of $\widetilde{\mathcal{W}}$.
Since Wulff shape $\mathcal{W}$ has the property (2) of Theorem \ref{main}, 
by Lemma \ref{leproperty}, it follows that
\[
P\in\bigcap_{Q\in \widetilde{\mathcal{W}}}H(Q)=\widetilde{\mathcal{W}}^{\circ}
\]
for any $P$ of $\widetilde{\mathcal{W}}$. 
Set $K=H(Q)$ is a hemisphere supports $\widetilde{\mathcal{W}}$ at $P$.
By duality, it follows that $Q$ is a point of $\partial\widetilde{\mathcal{W}}^\circ$ (Lemma 2.2 of \cite{hwam}).
Let $M$ be a point of the intersection of $PQ$ and the interior of $\widetilde{\mathcal{W}}$
, by Lemma \ref{lemmasmooth} such a point always exists.
Then it follows that 
\[
P\in \partial\widetilde{\mathcal{W}}\cap \Psi_{M}(\partial\widetilde{\mathcal{W}}^{\circ}).
\]
By assumption, this implies
\[
\Psi_{M}(P)=Q\in \partial\widetilde{\mathcal{W}}.
\]
Combining this fact and Lemma \ref{leproperty}, it follows that 
the hemisphere $H(P)$ supports $\widetilde{\mathcal{W}}$ at $Q$. 
This means 
the width of $\widetilde{\mathcal{W}}$ is not greater then thickness of the lune $H(P)\cap H(Q)$. 
Therefore, it follows that
\[
 \tag{$\star$}  
 \Delta(\widetilde{\mathcal{W}})\leq \Delta(H(P)\cap H(Q))=\pi-|PQ|=\pi/2.
\]
We prove the equality $ \Delta(\widetilde{\mathcal{W}})=\pi/2$ of ($\star$) always holds.
To apply proof by contradiction, assume that there exists hemisphere $K=H(Q_1)$ supports   
$\widetilde{\mathcal{W}}$ at $P_1$,
such that the width of $\widetilde{\mathcal{W}}$ determined by $K$ is smaller than $\pi/2$. 
This means there exists a point $Q_{2}$ of $\partial \widetilde{\mathcal{W}}^{\circ}$ such that 
the hemisphere $H(Q_2)$ supports $\widetilde{\mathcal{W}}$ and 
\[
{\rm width}_{K}(\widetilde{\mathcal{W}})=\Delta(H(Q_1)\cap H(Q_2))=\pi-|Q_{1}Q_{2}|<\frac{\pi}{2}.
\]
This implies $|Q_{1}Q_{2}|>\pi/2$.
Set $P_{2}$ be a point of the intersection of $\partial\widetilde{\mathcal{W}}$ and $\partial H(Q_{2})$ 
 Let $M_{1}$ be a point of the intersection of $P_{1}Q_{1}$ and interior of $\widetilde{\mathcal{W}}$,
 by Lemma \ref{lemmasmooth} such a point always exists. 
Then it follows that 
\[
P_{1}\in \partial\widetilde{\mathcal{W}}\cap \Psi_{M_{1}}(\partial\widetilde{\mathcal{W}}^{\circ}).
\] 
By assumption, it follows that
\[
\Psi_{M_{1}}(P_{1})=Q_{1}\in \partial\widetilde{\mathcal{W}}.
\]
In the same way, we have that
\[
\Psi_{M_{2}}(P_{2})=Q_{2}\in \partial\widetilde{\mathcal{W}}.
\]
where $M_{2}$ is a point of the intersection of $P_{2}Q_{2}$ and the interior of $\widetilde{\mathcal{W}}$. 
Therefore, we have the inequality 
\[
\pi/2<|Q_{1}Q_{2}|\leq {\rm diam(\mathcal{W})}.
\]
This contradicts Lemma \ref{leproperty}. 
Therefore, $\widetilde{\mathcal{W}}$ is a spherical convex body of constant width $\pi/2$.
\end{proof}
Since a Wulff shape in $\mathbb{R}^{n+1}$ is strictly convex and its boundary is $C^1$
diffeomorphic to $S^n$ if and only if its dual Wulff shape is strictly convex and the
boundary of it is $C^1$ diffeomorphic to $S^n$ (\cite{hnams}),
combining Theorem \ref{generallassak1} and Theorem \ref{main}, we have the following:
\begin{corollary}
Let $\mathcal{W}$ be a strictly convex Wulff shape and its boundary is $C^1$
diffeomorphic to $S^n$. Then the following statements are equivalent.
\begin{enumerate}
\item $\mathcal{W}$ is a self-dual Wulff shape.
\item The spherical Wulff shape $\widetilde{\mathcal{W}}=\alpha^{-1}(Id(\mathcal{W_\gamma}))$ of $\mathcal{W}$ has the property that 
$\Psi_{M}(P)$ is a point of $\partial\widetilde{\mathcal{W}}$
for any interior point $M$ of $\widetilde{\mathcal{W}}$ and 
any point $P$ of $\partial \widetilde{\mathcal{W}}\cap \Psi_M\left((\partial\widetilde{\mathcal{W}})^{\circ}\right)$.
\item The spherical dual Wulff shape $\widetilde{\mathcal{W}}^\circ$ of $\mathcal{W}$ has the property that 
$\Psi_{M}(P)$ is a point of $\partial\widetilde{\mathcal{W}}^\circ$
for any interior point $M$ of $\widetilde{\mathcal{W}}^\circ$ and 
any point $P$ of $\partial \widetilde{\mathcal{W}}^\circ\cap \Psi_M(\partial\widetilde{\mathcal{W}})$.
\item The spherical Wulff shape $\widetilde{\mathcal{W}}$ is of constant width $\pi/2$.
\item The spherical Wulff shape $\widetilde{\mathcal{W}}$ is of constant diameter $\pi/2$.
\end{enumerate}
\end{corollary}
{\bf Acknowledgements.}
This work was supported, in partial, by
Natural Science Basic Research Plan in Shaanxi Province of China
(2023-JC-YB-070),

\end{document}